\documentclass[a4paper,11pt]{article}
\usepackage[french, english]{babel}
\usepackage[latin1]{inputenc}
\usepackage{amsfonts}
\usepackage{amsmath}
\usepackage{amsthm}
\usepackage{amssymb}
\usepackage{mathrsfs}
\usepackage{ae, aecompl}
\usepackage{graphicx}
\usepackage[a4paper, portrait]{geometry}
\usepackage{sectsty}
\usepackage{lmodern}
\RequirePackage{mathrsfs}
\let\mathcal\mathscr
\usepackage{url}

\newcommand*{\Z}{\ensuremath{\mathbf{Z}}}                        
\newcommand*{\qp}{\ensuremath{\mathbf{Q}_p}}                     
\newcommand*{\zp}{\ensuremath{\mathbf{Z}_p}}

\newcommand*{\zpet}{\ensuremath{\mathbf{Z}_p^*}}
\newcommand*{\qpet}{\ensuremath{\mathbf{Q}_p^*}}
\DeclareMathOperator{\p1}{\bold{P}^1}

\newtheoremstyle{theorem}{11pt}{11pt}{\slshape}{}{\bfseries}{.}{.5em}{}
\newtheoremstyle{note}{11pt}{11pt}{}{}{\bfseries}{.}{.5em}{}

\theoremstyle{plain}
  \newtheorem{theorem}{Th\'{e}or\`{e}me}[section]
  \newtheorem{proposition}[theorem]{Proposition}
  \newtheorem{lemma}[theorem]{Lemme}
  \newtheorem{corollary}[theorem]{Corollaire}

\theoremstyle{definition}

\theoremstyle{remark}
  
  \newtheorem{remark}[theorem]{Remarque}

\begin{document}

\title{Equations diff\'{e}rentielles $p$-adiques et modules de Jacquet analytiques}
\author{Gabriel Dospinescu \thanks
  {\footnotesize {C.M.L.S, Ecole Polytechnique,
  gabriel.dospinescu@math.polytechnique.fr}}}

\maketitle

\begin{abstract}
\footnotesize
 Using differential techniques, we compute
the Jacquet module of the locally analytic vectors of irreducible
admissible unitary representations of ${\rm GL}_2(\qp)$.
 \end{abstract}

\selectlanguage{french}

\section{Introduction}

  Le but de cet article est d'\'{e}tudier le module de Jacquet des ${\rm GL}_2(\qp)$ repr\'{e}sentations de Banach unitaires admissibles, absolument irr\'{e}ductibles. On retrouve les r\'{e}sultats de \cite{C7}, mais les m\'{e}thodes sont
sensiblement diff\'{e}rentes.

\subsection{Notations}\label{not}

   On fixe une extension finie $L$ de $\qp$ et on note $\mathcal{R}$ l'anneau de Robba \`{a} coefficients dans $L$, i.e.
l'anneau des s\'{e}ries de Laurent $\sum_{n\in\mathbf{Z}} a_nT^n$, avec $a_n\in L$, qui convergent sur une couronne de type $0<v_p(T)\leq r$,
o\`{u} $r>0$ d\'{e}pend de la s\'{e}rie.

Soit $\chi: {\rm Gal}(\overline{\qp}/\qp)\to \zpet$ le caract\`{e}re cyclotomique. Il induit un isomorphisme
de $\Gamma={\rm Gal}(\qp(\mu_{p^{\infty}})/\qp)$ sur $\zpet$ et on note $a\to \sigma_a$ son inverse. On munit $\mathcal{R}$ d'une action de
$\Gamma$ et d'un Frobenius $\varphi$, en posant $(\sigma_a f)(T)=f((1+T)^a-1)$ et $(\varphi f)(T)=f((1+T)^p-1)$. Soit $\nabla=\lim_{a\to 1}\frac{\sigma_a-1}{a-1}$
l'action infinit\'{e}simale de $\Gamma$. Explicitement, on a $\nabla(f)=t(1+T)f'(T)$, o\`{u} $t=\log(1+T)$.

  Si $\delta:\qpet\to L^*$
est un caract\`{e}re continu, on note $\mathcal{R}(\delta)$ le $\mathcal{R}$-module libre
 de rang $1$ ayant une base $e$ (dite canonique ) telle que $\varphi(e)=\delta(p)e$ et $\sigma_a(e)=\delta(a)e$
pour tout $a\in\zpet$. On pose $w(\delta)=\delta'(1)$, la d\'{e}riv\'{e}e de $\delta$ en $1$ (rappelons que $\delta$ est automatiquement localement analytique).
Si $\delta$ est unitaire (i.e. si $\delta(\qpet)\subset O_{L}^*$), alors $w(\delta)$ est le poids de Hodge-Tate g\'{e}n\'{e}ralis\'{e} du caract\`{e}re de
${\rm Gal}(\overline{\qp}/\qp)$ attach\'{e} \`{a} $\delta$ par la th\'{e}orie locale du corps de classes, normalis\'{e}e de telle sorte que
$\chi$ corresponde \`{a} $x\to x\cdot |x|_p$.

\subsection{La correspondance de Langlands locale $p$-adique}\label{LLpadique}

  Une $L$-repr\'{e}sentation de ${\rm Gal}(\overline{\qp}/\qp)$ est un $L$-espace vectoriel de dimension finie, muni d'une action
$L$-lin\'{e}aire continue de ${\rm Gal}(\overline{\qp}/\qp)$. Les travaux de Fontaine, Cherbonnier-Colmez, Kedlaya et Berger (voir
\cite{Be1, CC, Fo1, Ke1}) associent
\`{a} une $L$-repr\'{e}sentation $V$ un $\mathcal{R}$-module $D_{\rm rig}=D_{\rm rig}(V)$ libre de rang $\dim_L V$, muni d'actions semi-lin\'{e}aires
de $\varphi$ et $\Gamma$, qui commutent. $D_{\rm rig}$ est aussi muni d'un inverse \`{a}  gauche $\psi$ de $\varphi$, qui
joue un grand r\^{o}le dans la th\'{e}orie. 

   Dans la suite de cette introduction on suppose que $V$ est de dimension $2$ sur $L$, absolument irr\'{e}ductible. La correspondance
de Langlands locale $p$-adique \cite{C5} associe \`{a} $V$ un $L$-espace de Banach $\Pi=\Pi(V)$, muni d'une action continue de ${\rm GL}_2(\qp)$,
qui en fait une repr\'{e}sentation unitaire, admissible et topologiquement absolument irr\'{e}ductible. Soit $\delta_D=\chi^{-1}\cdot \det V$,
que l'on voit comme caract\`{e}re de $\qpet$ et comme caract\`{e}re de ${\rm GL}_2(\qp)$, en composant avec le d\'{e}terminant. On peut utiliser
les actions de $\varphi,\psi$ et $\Gamma$ pour construire un faisceau ${\rm GL}_2(\qp)$-\'{e}quivariant sur $\p1(\qp)$, dont $D_{\rm rig}$
est l'espace des sections sur $\zp$. Par construction, $D_{\rm rig}^{\psi=0}$ est l'espace des sections sur $\zpet$ et l'application de restriction
\`{a} $\zpet$ est donn\'{e}e par ${\rm Res}_{\zpet}=1-\varphi\circ \psi$. Soit $w_D$ l'involution de $D_{\rm rig}^{\psi=0}$ 
d\'{e}crivant l'action de $\left(\begin{smallmatrix} 0 & 1 \\1 & 0\end{smallmatrix}\right)$. L'espace des sections globales 
du faisceau est donc 
$$D_{\rm rig}\boxtimes\p1=\{(z_1,z_2)\in D_{\rm rig}\times D_{\rm rig}|
\quad w_D({\rm Res}_{\zpet}(z_1))={\rm Res}_{\zpet}(z_2)\}.$$ On montre \cite[th. V.2.20]{C5}
que les vecteurs
localement analytiques $\Pi^{\rm an}$ de $\Pi$ vivent dans une suite exacte de ${\rm GL}_2(\qp)$-modules topologiques
 $$0\to (\Pi^{\rm an})^*\otimes\delta_D\to D_{\rm rig}\boxtimes \p1\to \Pi^{\rm an}\to 0.$$

\subsection{Repr\'{e}sentations triangulines}

  Dans \cite{C3}, Colmez d\'{e}finit un espace $\mathcal{S}_{\rm irr}$ de repr\'{e}sentations irr\'{e}ductibles de dimension $2$ de ${\rm Gal}(\overline{\qp}/\qp)$, appell\'{e}es triangulines. Un point de $\mathcal{S}_{\rm irr}$ est un triplet
   $s=(\delta_1,\delta_2,\mathcal{L})$, o\`{u} $\delta_1, \delta_2:\qpet\to L^*$
sont des caract\`{e}res continus et $\mathcal{L}\in \p1(L)$ si $\delta_1=x^k\chi\delta_2$ ($k\in\mathbf{N}$), ou 
$\mathcal{L}\in \mathbf{P}^0(L)=\{\infty\}$ si $\delta_1\notin \{x^k\chi\delta_2, \ k\in\mathbf{N}\}$. Si $s\in \mathcal{S}_{\rm irr}$, on note
$w(s)=w(\delta_1)-w(\delta_2)$, $V(s)$ la repr\'{e}sentation associ\'{e}e et $D_{\rm rig}(s)$ son $(\varphi,\Gamma)$-module sur
l'anneau de Robba. On note aussi $\Pi(s)=\Pi(V(s))$. Par construction, on a une suite exacte
 $0\to \mathcal{R}(\delta_1)\to D_{\rm rig}\to \mathcal{R}(\delta_2)\to 0$, dont la classe d'isomorphisme est d\'{e}termin\'{e}e par
 $\mathcal{L}$. L'espace $\mathcal{S}_{\rm irr}$ admet une partition
  $\mathcal{S}_{\rm irr}=\mathcal{S}_{*}^{\rm ng}\amalg \mathcal{S}_*^{\rm cris}\amalg \mathcal{S}_*^{\rm st}$, o\`{u}

  $\bullet$ $\mathcal{S}_*^{\rm cris}=\{s\in\mathcal{S}_{\rm irr}|
 w(s)\in\mathbb{N}^*, w(s)>v_p(\delta_1(p)) \quad \text{et}\quad
 \mathcal{L}= \infty\}$.

  $\bullet$ $\mathcal{S}_*^{\rm st}=\{s\in\mathcal{S}_{\rm irr}|
 w(s)\in\mathbb{N}^*, w(s)>v_p(\delta_1(p)) \quad \text{et}\quad
 \mathcal{L}\ne \infty\}$.

  $\bullet$ $\mathcal{S}_*^{\rm ng}=\{s\in\mathcal{S}_{\rm irr}| w(s)\notin\mathbb{N}^*\}$.

    Supposons que $\delta_1$ est localement alg\'{e}brique et que $V$ est une repr\'{e}sentation irr\'{e}ductible de dimension $2$.
    On d\'{e}montre alors \cite{C3} que:

    $\bullet$ $V$ correspond \`{a} un point de
   $\mathcal{S}_*^{\rm cris}$ si et seulement si $V$ devient
cristalline sur une extension ab\'{e}lienne de $\qp$;

  $\bullet$ $V$ correspond \`{a} un point de $\mathcal{S}_*^{\rm st}$ si et seulement si $V$ est une tordue par un caract\`{e}re d'ordre
 fini d'une repr\'{e}sentation semi-stable non cristalline.

\subsection{Le module de Jacquet analytique}

   Soit $U=\left(\begin{smallmatrix} 1 & \qp\\0 & 1\end{smallmatrix}\right)$. Si $\pi$ est une $L$-repr\'{e}sentation localement analytique de ${\rm GL}_2(\qp)$ (voir \cite{Em5, ST2, ST3} pour les bases de la th\'{e}orie), on note $J(\pi)$ son module
de Jacquet na\"{i}f, quotient de $\pi$ par l'adh\'{e}rence du sous-espace engendr\'{e} par les vecteurs $(u-1)\cdot v$, o\`{u}
$u\in U$ et $v\in \pi$. Le dual\footnote{Tous les duaux que l'on consid\`{e}re dans cet article sont topologiques.} de $J(\pi)$ est $J^*(\pi)=(\pi^*)^{U}$
et c'est naturellement une repr\'{e}sentation localement analytique du tore diagonal de ${\rm GL}_2(\qp)$.
 Le premier r\'{e}sultat est l'analogue $p$-adique d'un r\'{e}sultat classique
 de la th\'{e}orie des repr\'{e}sentations lisses, et confirme
le principe selon lequel les repr\'{e}sentations triangulines
correspondent aux ${\rm GL}_2(\qp)$-repr\'{e}sentations de la s\'{e}rie principale unitaire \cite{C6}.

 \begin{theorem}\label{th1}
   Soient $V$ et $\Pi$ comme dans \ref{LLpadique}. Alors
   $J^*(\Pi^{\rm an})$ est un
  $L$-espace vectoriel de dimension au plus $2$ et il est non nul si et seulement si
  $V$ est trianguline.
 \end{theorem}

  Ce th\'{e}or\`{e}me est aussi d\'{e}montr\'{e} dans \cite[th. 0.1]{C7}, en utilisant l'action de $\varphi$ sur $D_{\rm rig}$. Notre approche est orthogonale
  (elle utilise l'action infinit\'{e}simale de $\Gamma$ au lieu de celle de $\varphi$) et plus directe: si $u^+$
d\'{e}signe l'action infinit\'{e}simale de $U$, le r\'{e}sultat principal de
 \cite{D1} montre que le noyau de $u^+$ sur
  l'espace $(\Pi^{\rm an})^{*}$ s'identifie \`{a}
l'espace des solutions de l'\'{e}quation diff\'{e}rentielle $(\nabla-a)(\nabla-b)z=0$, o\`{u}
$z\in D_{\rm rig}$, $a$ et $b$ sont les poids de Hodge-Tate g\'{e}n\'{e}ralis\'{e}s de $V$
et $\nabla$ est l'action infinit\'{e}simale de $\Gamma$ sur $D_{\rm rig}$. Cela ram\`{e}ne l'\'{e}tude de
$J^*(\Pi^{\rm an})$ \`{a} la r\'{e}solution de cette \'{e}quation diff\'{e}rentielle, ce qui se fait sans mal.

  Soit $\delta_1\otimes\delta_2$ le caract\`{e}re $(a,d)\to \delta_1(a)\delta_2(d)$ du tore
  diagonal $T$ de $\rm{GL}_2(\qp)$. Le r\'{e}sultat suivant pr\'{e}cise le th\'{e}or\`{e}me \ref{th1}
  et correspond \`{a} \cite[th.~0.6]{C7}.

\begin{theorem}\label{th4}

   Soit $s=(\delta_1,\delta_2,\mathcal{L})\in\mathcal{S}_{\rm irr}$.

1) Si $s\in \mathcal{S}_*^{\rm st}$ ou si $w(s)\notin {\Z}^*$, alors
 $J^*(\Pi^{\rm an}(s))=\delta_1^{-1}\otimes \delta_2^{-1}\chi$.

2) Si $w(s)\in \{...,-2,-1\}$, alors
$J^*(\Pi^{\rm an}(s))=(\delta_1^{-1}\otimes \delta_2^{-1}\chi)\oplus
 (x^{w(s)}\delta_1^{-1}\otimes x^{-w(s)}\delta_2^{-1}\chi)$.

3) Si $s\in\mathcal{S}_*^{\rm cris}$, alors $J^*(\Pi^{\rm an}(s))=(\delta_1^{-1}\otimes \delta_2^{-1}\chi)\oplus
 (x^{-w(s)}\delta_2^{-1}\otimes x^{w(s)}\delta_1^{-1}\chi)$ si $s$ est non exceptionnel (i.e. si
 $\delta_1\ne x^{w(s)}\delta_2$) et $J^*(\Pi^{\rm an}(s))=(\delta_1^{-1}\otimes \delta_2^{-1}\chi)\otimes
 \left(\begin{smallmatrix} 1 & v_p(a/d) \\0 & 1\end{smallmatrix}\right)$ dans le cas contraire.

\end{theorem}

\subsection{L'involution $w_D$ et d\'{e}vissage de $D_{\rm rig}\boxtimes \p1$}\label{invdevis}

  On suppose que $s=(\delta_1,\delta_2,\mathcal{L})\in\mathcal{S}_{\rm irr}$ et on note
$e_i$ la base canonique de $\mathcal{R}(\delta_i)$, $p_s$ la projection canonique $D_{\rm rig}(s)\to \mathcal{R}(\delta_2)$
et $\hat{e}_2\in D_{\rm rig}(s)$ tel que $p_s(\hat{e}_2)=e_2$. Si $U$ est un ouvert compact de $\zp$, soit
${\rm LA}(U)$ l'espace des fonctions localement analytiques sur $U$ \`{a} valeurs dans $L$ et soit $\mathcal{D}(U)$ son
dual topologique. Si $\delta:\qpet\to L^*$ est un carac\`{e}re continu,
on d\'{e}finit une involution $w_{\delta}$ sur $\mathcal{D}(\zpet)$
 en demandant que $$\int_{\zpet} \phi (w_{\delta} \mu)=\int_{\zpet} \delta(x)\phi\left(\frac{1}{x}\right)$$
pour tout $\phi\in {\rm LA}(\zpet)$. Via l'isomorphisme $(\mathcal{R}^{+})^{\psi=0}\simeq \mathcal{D}(\zpet)$
donn\'{e} par le th\'{e}or\`{e}me d'Amice, cela induit une involution $w_{\delta}$ sur
$(\mathcal{R}^+)^{\psi=0}$, qui satisfait $w_{\delta}(\sigma_a f)=\delta(a)\sigma_{\frac{1}{a}}(w_{\delta}(f))$.
Cette involution s'\'{e}tend de mani\`{e}re unique en une involution de $\mathcal{R}^{\psi=0}$, satisfaisant la m\^{e}me relation
que ci-dessus. Le r\'{e}sultat suivant fournit une description
 plus ou moins explicite de l'involution $w_D$ dans le cas triangulin.
 Si $s\in \mathcal{S}_{*}^{\rm cris}$, cela permet de retrouver et renforcer le d\'{e}licat lemme II.3.13 de \cite{C5}.

\begin{theorem}\label{th2}

   Pour tout $f\in \mathcal{R}^{\psi=0}$ on a $w_{D(s)}(f\cdot e_1)=\delta_1(-1)w_{\delta_D\cdot\delta_1^{-2}}(f)\cdot e_1$
et $$p_s(w_{D(s)}(f\cdot \varphi(\hat{e}_2)))=\delta_2(-1)w_{\delta_D\cdot\delta_2^{-2}}(f)\cdot\varphi(e_2).$$
\end{theorem}

   La tr\`{e}s mauvaise convergence de la suite\footnote{ Soit $D^{\dagger}$ le sous-module surconvergent de $D_{\rm rig}$
   et $D$ le $(\varphi,\Gamma)$-module sur le corps de Fontaine $\mathcal{E}$ attach\'{e} \`{a} $V$. Il d\'{e}coule de la correspondance
   de Langlands locale $p$-adique pour ${\rm GL}_2(\qp)$ (\cite{C5}, th. II.3.1, prop. V.2.1 et lemme V.2.4) que si $z\in D^{\dagger,\psi=0}$, la suite $$\sum_{i\in \left(\mathbb{Z}/p^n\mathbb{Z}\right)^*}
       \delta_D(i^{-1})(1+T)^i \sigma_{-i^2} \varphi^n \psi^n ((1+T)^{-i^{-1}}z)$$ converge dans $D$ (mais pas dans $D^{\dagger}$) et sa limite
   $w_D(z)$ appartient \`{a} $D^{\dagger,\psi=0}$. L'extension de $w_D$ \`{a} $D_{\rm rig}^{\psi=0}$ se fait en utilisant 
la densit\'{e} de $D^{\dagger}$ dans $D_{\rm rig}$.} d\'{e}finissant $w_D$ rend d\'{e}licate une preuve directe du th\'{e}or\`{e}me
\ref{th2}.
On d\'eduit du th\'{e}or\`{e}me \ref{th2} le corollaire \ref{devi} ci-dessous
qui est le point de d\'{e}part \cite{C7, LXZ} pour l'\'{e}tude
des vecteurs localement analytiques de la s\'{e}rie principale unitaire. 

 Soit $$\mathcal{R}\boxtimes_{\delta}\p1=\{(f_1,f_2)\in\mathcal{R}\times\mathcal{R}|
   \quad
{\rm Res}_{\zpet}(f_2)=w_{\delta}({\rm Res}_{\zpet}(f_1))\}.$$ En copiant les constructions
de Colmez, on munit $\mathcal{R}\boxtimes_{\delta}\p1$ d'une structure naturelle
de ${\rm GL}_2(\qp)$-module topologique (pour les d\'{e}tails voir \ref{delta}).
Ce module $\mathcal{R}\boxtimes_{\delta}\p1$ est \'{e}troitement li\'{e} aux induites paraboliques\footnote{On note dans la suite ${\rm Ind}_B^{G}(\delta_1\otimes\delta_2)$
l'espace des fonctions localement analytiques $f:{\rm GL}_2(\qp)\to L$ telles que 
$f\left( \left(\begin{smallmatrix} a & b \\0 & d\end{smallmatrix}\right)g\right)=\delta_1(a)\delta_2(d)f(g)$ pour tous
$a,d\in \qpet$, $b\in\qp$ et $g\in {\rm GL}_2(\qp)$.}, car on peut montrer qu'il vit dans une suite
exacte de ${\rm GL}_2(\qp)$-modules topologiques
$$0\to ({\rm Ind}_B^{G}(\delta^{-1}\otimes 1))^*
\to \mathcal{R}\boxtimes_{\delta}\p1\to {\rm Ind}_B^{G}(\chi^{-1}\delta\otimes\chi^{-1})\to 0.$$

\begin{corollary}\label{devi}
  Soit $s\in \mathcal{S}_{\rm irr}$. La suite exacte $0\to \mathcal{R}(\delta_1)\to
  D_{\rm rig}\to \mathcal{R}(\delta_2)\to 0$ induit une suite exacte de ${\rm GL}_2(\qp)$-modules
  topologiques
   $$0\to (\mathcal{R}\boxtimes_{\delta_D\cdot\delta_1^{-2}}\p1)\otimes\delta_1\to 
D_{\rm rig}\boxtimes\p1\to (\mathcal{R}\boxtimes_{\delta_D\cdot
    \delta_2^{-2}}\p1)\otimes\delta_2\to 0.$$
\end{corollary}

    Ce r\'{e}sultat est d\'{e}montr\'{e} dans \cite[th. 4.6]{C7}, ainsi que dans \cite{LXZ} (prop. 6.8)
    par des m\'{e}thodes diff\'{e}rentes. Il permet de donner \cite[th.07]{C7} 
une description compl\`{e}te de la repr\'{e}sentation $\Pi^{\rm an}$ (en particulier,
de montrer qu'elle est de longueur finie et d'en trouver les constituants de Jordan-H\H{o}lder), confirmant ainsi des conjectures
de Berger, Breuil et Emerton \cite{BB, Em1}.

  \subsection{Remerciements}
  Ce travail est une partie de ma th\`{e}se de doctorat, r\'{e}alis\'{e}e
 sous la direction de Pierre Colmez et de Ga\"{e}tan Chenevier.
 Je leur suis profond\'{e}ment reconnaissant
 pour des nombreuses discussions que nous avons eues et pour leurs tout aussi nombreuses suggestions.
 Je voudrais aussi remercier R.Taylor, X.Caruso et L. Berger pour m'avoir invit\'{e} \`{a} exposer ces r\'{e}sultats \`{a} l'I.A.S, dans
 le cadre du Workshop on Galois Representations and Automorphic Forms, et \`{a} l'E.N.S Lyon, dans le cadre de la
 conf\'{e}rence "Th\'{e}orie de Hodge p-adique, \'{e}quations diff\'{e}rentielles p-adiques et leurs applications".
 Merci aussi \`{a} Andrea Pulita pour une discussion qui m'a permis de simplifier une d\'{e}monstration et \`{a} Ramla Abdellatif, qui
m'a grandement aid\'{e} \`{a} am\'{e}liorer la r\'{e}daction.

\section{Le module de Jacquet de $\Pi^{\rm an}$}

\subsection{Un r\'{e}sultat de finitude}\label{finiteness}

  \quad On d\'{e}montre un r\'{e}sultat de finitude pour les \'{e}quations diff\'{e}rentielles $p$-adiques attach\'{e}es
   aux repr\'{e}sentations galoisiennes. Rappelons \cite[V.1]{Be1} que si $V$ est une $L$-repr\'{e}sentation de ${\rm Gal}(\overline{\qp}/\qp)$
et si $D_{\rm rig}$ est son $(\varphi,\Gamma)$-module sur $\mathcal{R}$, alors l'action de $\Gamma$ peut se d\'{e}river, d'o\`{u} une connexion
$\nabla=\lim_{a\to 1}\frac{\sigma_a-1}{a-1}$ sur $D_{\rm rig}$, au-dessus de la connexion $\nabla$ sur $\mathcal{R}$ (introduite dans 
\ref{not}). Si $P\in L[X]$, notons
   $$D_{\rm rig}^{P(\nabla)=0}=\{z\in D_{\rm rig}|
P(\nabla)(z)=0\}.$$

\begin{proposition}\label{fini}
  Soit $V$ une $L$-repr\'{e}sentation quelconque de ${\rm Gal}(\overline{\qp}/\qp)$
  et soit $D_{\rm rig}$ son $(\varphi,\Gamma)$-module sur $\mathcal{R}$.
  Si $P\in L[X]$ est non nul, alors $$\dim_{L} D_{\rm rig}^{P(\nabla)=0}\leq \dim_L (V)\cdot \deg(P).$$
\end{proposition}

\begin{proof} Quitte \`{a} remplacer $L$ par une extension finie, on peut supposer que toutes les racines
de $P$ dans $\overline{\qp}$ sont dans $L$. On d\'{e}montre le th\'{e}or\`{e}me par r\'{e}currence sur $\deg P$.
Pour traiter le cas $\deg(P)=1$, nous avons besoin de quelques pr\'{e}liminaires.

\begin{lemma}\label{primo}
  On a $({\rm Frac}(\mathcal{R}))^{\nabla=0}=L$.
 \end{lemma}

 \begin{proof}
 Ce r\'{e}sultat est probablement standard, mais faute d'une r\'{e}f\'{e}rence voici une preuve.
 Rappelons que $\nabla(f)=t\cdot (1+T)f'(T)$ pour $f\in\mathcal{R}$, o\`{u} $t=\log(1+T)$. En particulier
 $\mathcal{R}^{\nabla=0}=L$. La condition $\nabla\left(\frac{f}{g}\right)=0$ \'{e}quivaut \`{a} $f'\cdot g=f\cdot g'$. Soit $r>0$ tel que
 $f$ et $g$ soient analytiques sur la couronne $0<v_p(T)\leq r$. Si $0<r'\leq r$ et si $z\in \mathbf{C}_p$ est un z\'{e}ro
 de $g$ dans la couronne $r'\leq v_p(T)\leq r$, la relation $f'\cdot g=g'\cdot f$ montre que $z$ a la m\^{e}me multiplicit\'{e}
 (finie) dans $f$ et dans $g$. D'apr\`{e}s des r\'{e}sultats standard de Lazard (voir par exemple la prop. 4.12 de \cite{Be1}), le quotient
 $\frac{f}{g}$ est donc analytique dans la couronne $r'\leq v_p(T)\leq r$. Comme cela vaut pour tout $r'\leq r$, le quotient
 $\frac{f}{g}$ est analytique sur la couronne $0<v_p(T)\leq r$ et donc est dans $\mathcal{R}^{\nabla=0}=L$.

 \end{proof}

   Le cas $\deg (P)=1$ suit alors du lemme ci-dessus et du fait que le rang de $D_{\rm rig}$ sur $\mathcal{R}$ est
   $\dim_L(V)$.

 \begin{lemma}\label{secondo}
  Soit $\alpha\in L$. L'application naturelle $D_{\rm rig}^{\nabla=\alpha}\otimes_{L} \mathcal{R}\to
  D_{\rm rig}$ est injective.
 \end{lemma}

\begin{proof}
  Il s'agit de v\'{e}rifier que si $z_1,z_2,...,z_d\in D_{\rm rig}^{\nabla=\alpha}$ sont libres sur $L$, alors ils sont
  libres sur $\mathcal{R}$. Soit (quitte \`{a} renum\'{e}roter les $z_i$) $\sum_{i=1}^{k} f_i\cdot z_i=0$ une relation
  de longueur minimale sur $\mathcal{R}$ et soit $g_i=\frac{f_i}{f_1}\in {\rm Frac}(\mathcal{R})$. En appliquant
  $\nabla$ et en utilisant le fait que $\nabla(z_i)=\alpha\cdot z_i$, on obtient $\sum_{i=1}^{k} \nabla(g_i)\cdot z_i=0$.
  Comme $\nabla(g_1)=0$, par minimalit\'{e} on obtient $\nabla(g_i)=0$ pour tout $i$. Le lemme \ref{primo} permet alors de conclure.

\end{proof}

Supposons le th\'{e}or\`{e}me d\'{e}montr\'{e} pour $\deg P=n$ et montrons-le pour
$\deg P=n+1$. Soit $P=(X-\alpha)\cdot Q(X)$, avec $Q\in L[X]$ et notons $W_1=D_{\rm rig}^{P(\nabla)=0}$ et $W_2=D_{\rm rig}^{Q(\nabla)=0}$. Alors $\nabla-\alpha$ est un op\'{e}rateur $L$-lin\'{e}aire de $W_1$ dans $W_2$, dont le noyau est de dimension au plus $\dim_L(V)$ (lemme \ref{secondo})
et dont l'image est de dimension au plus $\dim_L W_2\leq \deg (Q)\cdot \dim_L V$. Le r\'{e}sultat s'en d\'{e}duit.

\end{proof}

\subsection{Finitude et annulation du module de Jacquet}\label{mod}

 Dans la suite on suppose que
$V$ et $\Pi$ sont comme dans \ref{LLpadique} (donc $V$ est de dimension $2$). Pour les autres notations
utilis\'{e}es dans la suite (en particulier $\nabla$ et $D_{\rm rig}\boxtimes\p1$), voir \ref{LLpadique} et \ref{finiteness}.

 L'ingr\'{e}dient essentiel pour l'\'{e}tude de $J^*(\Pi^{\rm an})$ est le r\'{e}sultat suivant, dans
 lequel $u^+$ d\'{e}signe
l'action infinit\'{e}simale de l'unipotent sup\'{e}rieur $U$ de ${\rm GL}_2(\qp)$.
Soient $a$ et $b$ les poids de Hodge-Tate g\'{e}n\'{e}ralis\'{e}s de $V$.

\begin{proposition} \label{inf}
  Pour tout $z=(z_1,z_2)\in D_{\rm rig}\boxtimes \p1$ on a
$$ u^+(z)= \left(tz_1, -\frac{(\nabla-a)(\nabla-b)z_2}{t}\right).$$
\end{proposition}

\begin{proof}
 On peut \'{e}crire $z=z_1+w\cdot (\varphi\circ \psi(z_2))$, o\`{u} $w=\left(\begin{smallmatrix} 0 & 1 \\1 & 0\end{smallmatrix}\right)$
et o\`{u} l'on voit $D_{\rm rig}$ comme sous-espace de $D_{\rm rig}\boxtimes \p1$ comme dans \ref{LLpadique}. On a donc
$$u^+(z)=u^+(z_1)+w\cdot u^-(\varphi(\psi(z_2)))$$ et le r\'{e}sultat d\'{e}coule alors de \cite[th. 1]{D1}.

\end{proof}

   D'apr\`{e}s la proposition \ref{fini}, le $L$-espace vectoriel $$X=\{z\in D_{\rm rig}| (\nabla-a)(\nabla-b)z=0\}$$ 
est de dimension au plus $4$ (on fera mieux par la suite). Comme $X$ est stable par $\varphi$, on a 
$X\subset \varphi(D_{\rm rig})$, de telle sorte que $(0,z)\in D_{\rm rig}\boxtimes\p1$ pour tout $z\in X$. 
Rappelons aussi que les vecteurs localement analytiques $\Pi^{\rm an}$ de $\Pi$ vivent dans une suite
exacte de ${\rm GL}_2(\qp)$-modules topologiques
  $$0\to (\Pi^{\rm an})^*\otimes\delta_D\to D_{\rm rig}\boxtimes\p1\to \Pi^{\rm an}\to 0.$$

\begin{proposition}\label{annulation} On a une \'{e}galit\'{e} de sous-$L$-espaces vectoriels de $D_{\rm rig}\boxtimes\p1$
  $$J^*(\Pi^{\rm an})\otimes\delta_D=\{(0,z)|z\in X\}.$$ En particulier, $J^*(\Pi^{\rm an})$ est de dimension au plus 
$4$ sur $L$.
\end{proposition}

\begin{proof}

 Soit $\pi=\Pi^{\rm an}$. L'inclusion
  $\pi^*\otimes\delta_D\subset D_{\rm rig}\boxtimes \p1$ induit une inclusion
  $$J^*(\pi)\otimes\delta_D\subset (D_{\rm rig}\boxtimes\p1)^{U}\subset
  (D_{\rm rig}\boxtimes \p1)^{u^+=0}$$
   et, d'apr\`{e}s la proposition \ref{inf}, on a
   $$(D_{\rm rig}\boxtimes \p1)^{u^+=0}=\{(0,z)| z\in X\}.$$
La proposition
   \ref{fini} montre alors que $\dim_L J^*(\pi)\leq 4$. 

 Montrons maintenant que les inclusions pr\'{e}c\'{e}dentes sont
des \'{e}galit\'{e}s. Nous aurons besoin du r\'{e}sultat suivant, qui montre en particulier
que $(D_{\rm rig}\boxtimes\p1)^{U}=
  (D_{\rm rig}\boxtimes \p1)^{u^+=0}$.

\begin{lemma}\label{ut}
 Soit $M$ une $L$-repr\'{e}sentation localement analytique de $ \left(\begin{smallmatrix} p^{\mathbf{Z}} & \qp \\0 & 1\end{smallmatrix}\right)$.
Si $M^{u^+=0}$ est de dimension finie sur $L$, alors $M^{u^+=0}=M^U$.
\end{lemma}

\begin{proof}
 Il existe $n$ tel que $M^{u^+=0}$ soit invariant par 
$\left(\begin{smallmatrix} 1 & p^n\zp\\0 & 1\end{smallmatrix}\right)$. Si $m\in M^{u^+=0}$ et
$a\in \qp$, on a alors pour tout $k\geq n-v_p(a)$
  $$ \left(\begin{smallmatrix} p^k & 0 \\0 & 1\end{smallmatrix}\right) \left(\begin{smallmatrix} 1 & a \\0 & 1\end{smallmatrix}\right)m=
 \left(\begin{smallmatrix} 1 & p^ka \\0 & 1\end{smallmatrix}\right) \left(\begin{smallmatrix} p^k & 0 \\0 & 1\end{smallmatrix}\right)m=
 \left(\begin{smallmatrix} p^k & 0 \\0 & 1\end{smallmatrix}\right)m,$$ donc $ \left(\begin{smallmatrix} 1 & a \\0 & 1\end{smallmatrix}\right)m=m$.
Cela permet de conclure.
\end{proof}

Pour finir la preuve de la proposition \ref{annulation}, il reste
  \`{a} voir que $(0,z)\in \pi^*\otimes\delta_D$ pour tout $z\in X$.
Les $\varphi^n(z)$ vivent dans $X$, qui est un $L$-espace vectoriel de dimension
 finie. Cela entra\^{i}ne (voir \cite[prop 3.2]{Be1}) que $z\in \tilde{D}_{\rm rig}^+$, o\`{u}\footnote{Rappelons que  $\mathbf{\tilde{B}}_{\rm rig}^+=\cap_{n\geq 1}
    \varphi^n(\mathbf{B}_{\rm cris}^+)$.}
$\tilde{D}_{\rm rig}^+=(\mathbf{\tilde{B}}_{\rm rig}^+\otimes_{\qp} V)^{\rm {Ker}\chi}$. Mais d'apr\`{e}s
\cite[lemme V.2.17]{C5}, le module $\tilde{D}_{\rm rig}^+$ est inclus dans $\pi^*\otimes
    \delta_D$. Le r\'{e}sultat en d\'{e}coule (noter que $(0,z)=w\cdot z$). 
\end{proof}

\begin{corollary}
  Si $V$ n'est pas trianguline, alors $J^*(\Pi^{\rm an})=0$.
\end{corollary}

\begin{proof}
  Il suffit de v\'{e}rifier que $X=0$ si $V$ n'est pas trianguline. 
Si $X\ne 0$, il existe (apr\`{e}s avoir remplac\'{e}
    $L$ par une extension finie) un vecteur propre
   pour $\varphi$ et $\Gamma$ dans $X$. On en d\'{e}duit que $V$ est trianguline (utiliser
   le lemme 3.2 de \cite{C3}).
\end{proof}

\subsection{Le module de Jacquet dans le cas triangulin}\label{conv2}

 \textit{Dans la suite on fixe un point $s=(\delta_1,\delta_2,\mathcal{L})\in \mathcal{S}_{\rm irr}$
et on note $V=V(s)$, $\Pi=\Pi(s)$, etc}. On note $e_i$ la base canonique de $\mathcal{R}(\delta_i)$ (rappelons que
$\varphi(e_i)=\delta_i(p)e_i$ et $\sigma_a(e_i)=\delta_i(a)e_i$) et $p_s$ la projection de $D_{\rm rig}$ sur
$\mathcal{R}(\delta_2)$. Noter que les poids de Hodge-Tate g\'{e}n\'{e}ralis\'{e}s de $V$ sont
$w(\delta_1)$ et $w(\delta_2)$ (\cite[prop 4.5]{C3}). De plus, comme $D_{\rm rig}$ est de pente $0$ et que $V$ est irr\'{e}ductible,
la th\'{e}orie des pentes de Kedlaya montre que $v_p(\delta_1(p))=-v_p(\delta_2(p))>0$.

  Le but de cette partie est de d\'{e}montrer le th\'{e}or\`{e}me \ref{th4} de l'introduction. Ce th\'{e}or\`{e}me
d\'{e}coule de la proposition \ref{annulation} et de la proposition \ref{prop2} ci-dessous, qui d\'{e}termine l'espace $X$ (d\'{e}fini 
dans \ref{mod}). Cela va demander quelques pr\'{e}liminaires.

  Rappelons que $s\in \mathcal{S}_{*}^{\rm cris}$ est dit exceptionnel si $\delta_1=x^{w(s)}\delta_2$. Si $s\in\mathcal{S}_{*}^{\rm cris}$,
  il existe $e_2'\in D_{\rm rig}$ tel que $p_s(e_2')=t^{w(s)}e_2$ et $\sigma_a(e_2')=a^{w(s)}\delta_2(a)e_2'$ pour tout $a\in \zpet$.
  Si $s$ n'est pas exceptionnel, on peut choisir $e_2'$ tel que $\varphi(e_2')=p^{w(s)}\delta_2(p)e_2'$ et alors
  $e_2'$ est unique \`{a} multiplication par un \'{e}l\'{e}ment de $L^*$ pr\`{e}s. Si $s$ est exceptionnel,
  on peut choisir $e_2'$ tel que $\varphi(e_2')=p^{w(s)}\delta_2(p)e_2'+e_1$, et alors $e_2'$ est unique \`{a} addition
  pr\`{e}s d'un \'{e}l\'{e}ment de $Le_1$. Ces r\'{e}sultats sont d\'{e}duits du calcul de la cohomologie de
  $\mathcal{R}(\delta)$, voir \cite[prop. 3.10]{C3}. On aura besoin dans la suite du r\'{e}sultat suivant
   \cite[lemme 3.22]{C7}:
   
   \begin{lemma}\label{propre} Soit $V$ une $L$-repr\'{e}sentation irr\'{e}ductible de dimension $2$ de ${\rm Gal}(\overline{\qp}/\qp)$ et
   soit $D_{\rm rig}$ son $(\varphi,\Gamma)$-module. 
   
   a) Si $D_{\rm rig}$ poss\`{e}de un vecteur propre pour les actions de $\varphi$ et $\Gamma$, alors $V$ est trianguline.
   
   b) Si $V$ correspond \`{a} un point $s\in \mathcal{S}_{*}^{\rm st}\cup \mathcal{S}_{*}^{\rm ng}$ ou si
   $s\in \mathcal{S}_{*}^{\rm cris}$ est exceptionnel, alors les vecteurs propres pour l'action
   de $\varphi$ et $\Gamma$ sont dans $\cup_{k\geq 0} L^*\cdot t^ke_1$. 
   
   c) Si $V$ correspond \`{a} un point $s\in \mathcal{S}_{*}^{\rm cris}$ non exceptionnel, les
   vecteurs propres pour $\varphi$ et $\Gamma$ sont dans $L^*\cdot t^k e_1$ ou 
  dans $L^*\cdot t^ke_2'$ pour un $k\geq 0$.
   
   \end{lemma}

   Le lemme suivant sera utilis\'{e} constamment dans la suite. Il fournit aussi une 
d\'{e}monstration tr\`{e}s directe de la proposition 1.19 de \cite{C7}.

   \begin{lemma} \label{eqdiff}
 Soit $k\in L$. L'espace des solutions de l'\'{e}quation $\nabla f+kf=0$ (avec
 $f\in\mathcal{R}$) est $\{0\}$ si $k\notin\{0,-1,-2,...\}$ et $L\cdot t^{-k}$
 si $k\in\{0,-1,-2,...\}$.
\end{lemma}

\begin{proof}
Soit $f\in \mathcal{R}$ une solution non nulle de l'\'{e}quation $\nabla f+kf=0$
  et soit $j$ le plus grand entier positif tel que $f\in t^j\cdot \mathcal{R}$.
  Posons $f=t^j\cdot g$, avec $g\in\mathcal{R}-t\cdot\mathcal{R}$. On a
  $\nabla g+(k+j)g=0$. Comme $\nabla(\mathcal{R})\subset t\cdot \mathcal{R}$, on obtient
 $k+j=0$ et $g\in L$. Le r\'{e}sultat s'en d\'{e}duit.

\end{proof}

  Notons $X_2=\{z\in D_{\rm rig}| (\nabla-w(\delta_2))z=0\}$, de telle sorte que
  $(\nabla-w(\delta_1))X\subset X_2$.

 \begin{lemma}\label{prec}
  On a $X_2=0$ si $w(s) \notin \{0,-1,-2,...\}$ et $X_2=L\cdot t^{-w(s)}e_1$ si $w(s)\in\{0,-1,-2,...\}$.

 \end{lemma}

\begin{proof}

  Comme $\nabla e_j=w(\delta_j)\cdot e_j$, l'\'{e}quation $(\nabla-w(\delta_2))(fe_1)=0$ \'{e}quivaut \`{a} $\nabla f+w(s)f=0$
  et l'\'{e}quation
  $(\nabla-w(\delta_2))(fe_2)=0$ \'{e}quivaut \`{a} $\nabla f=0$ et donc \`{a} $f\in L$.

Supposons que $w(s)\notin\{0,-1,...\}$. La suite exacte
 $0\to\mathcal{R}(\delta_1)\to D_{\rm rig}\to \mathcal{R}(\delta_2)\to 0$
 et l'observation du premier paragraphe montrent que $X_2$ s'injecte dans
 $Le_2$. Supposons que $X_2\ne 0$, donc $\dim_L X_2=1$. Soit $x\in X_2-\{0\}$, donc
 $x\notin\mathcal{R}e_1$ et $x$ est vecteur propre
 pour $\varphi$ et $\Gamma$. Le lemme \ref{propre} montre que
 $s\in \mathcal{S}_{*}^{\rm cris}$ n'est pas exceptionnel et qu'il existe $c\in L^*$ et $k\in\mathbb{N}$ tels que
 $x=ct^k\cdot e_2'$.
  Comme $x\in X_2$ et $\nabla e_2'=(w(s)+w(\delta_2))e_2'$, on
obtient $w(s)=-k\in \{0,-1,-2,...\}$, une contradiction. Donc $X_2=0$.

  Supposons maintenant que $w(s)\in\{0,-1,-2,...\}$. Le premier paragraphe
  nous fournit une suite exacte $0\to Lt^{-w(s)}e_1\to X_2\to Le_2$ et il reste \`{a} voir
  qu'elle n'est pas exacte \`{a} droite. Supposons donc qu'il existe $z\in X_2$ qui s'envoie sur $e_2$.  Alors
$\varphi(z)-\delta_2(p)z$ est dans
$X_2\cap \mathcal{R}e_1=Lt^{-w(s)}e_1$ et donc, quitte \`{a} travailler\footnote{Cela
utilise le fait que $p^{-w(s)}\delta_1(p)\ne \delta_2(p)$, car $v_p(\delta_1(p))>0$ et $v_p(\delta_2(p))<0$.} avec $z+ct^{-w(s)}e_1$ pour un $c\in L$ bien choisi, on peut supposer
 que $\varphi(z)=\delta_2(p)z$. Soit $\gamma\in\Gamma$. Il
existe $a\in L$ tel que $\gamma(z)=\delta_2(\chi(\gamma))z+at^{-w(s)}e_1$. Comme $\varphi$ et
$\Gamma$ commutent, un petit calcul donne $a(p^{-w(s)}\delta_1(p)-\delta_2(p))=0$, donc $a=0$. Mais alors la suite exacte
$0\to\mathcal{R}(\delta_1)\to D_{\rm rig}\to\mathcal{R}(\delta_2)\to 0$
est scind\'{e}e, contradiction.
\end{proof}

 \begin{proposition} \label{prop2} 1) Si $s\in\mathcal{S}_*^{\rm st}$ ou si $w(s)\notin {\Z}^*$,
   alors $X=L\cdot e_1$.

  2) Si $w(s)\in\{...,-2,-1\}$, alors $X=L\cdot e_1\oplus L\cdot t^{-w(s)}e_1$.

  3) Si $s\in\mathcal{S}_*^{\rm cris}$, alors $X=L\cdot e_1\oplus L\cdot e_2'$.

\end{proposition}

\begin{proof} Commen\c{c}ons par remarquer que $e_1\in X$ et que $t^{-w(s)}e_1\in X$ (resp.
$e_2'\in X$) pour $s\in \{...,-2,-1\}$ (resp. $s\in\mathcal{S}_*^{\rm cris}$). Pour l'inclusion inverse, on va distinguer deux cas:

$\bullet$  Si $w(s)\notin\{0,-1,-2,...\}$, la proposition pr\'{e}c\'{e}dente montre que
 $(\nabla-w(\delta_1))z=0$ si $z \in X$, donc\footnote{Noter que $(\nabla-w(\delta_1))(fe_2)=(\nabla-w(s))f\cdot e_2$.}
 la triangulation de
 $D_{\rm rig}$ induit une suite exacte
 $$0\to Le_1\to X\to \{fe_2| \nabla(f)=w(s)f\}.$$ Si $w(s)\notin\mathbf{N}^*$, le lemme
 \ref{eqdiff} montre que le terme de droite de cette suite exacte est nul et $X=Le_1$.
 Si $w(s)\in\mathbf{N}^*$ (ce qui inclut les cas $s\in\mathcal{S}_{*}^{\rm cris}$ et
 $s\in \mathcal{S}_{*}^{\rm st}$), on obtient donc (par le lemme \ref{eqdiff}) une suite exacte
 $0\to Le_1\to X\to L\cdot t^{w(s)}$. Cela montre d\'{e}j\`{a} que $\dim_L (X)\leq 2$ et permet
 de conclure dans le cas $s\in \mathcal{S}_{*}^{\rm cris}$. Supposons que $s\in \mathcal{S}_{*}^{\rm st}$
 et montrons que cette suite n'est pas exacte \`{a} droite (et donc que $X=L\cdot e_1$). Si ce n'\'{e}tait pas le cas,
 on trouve comme dans la preuve du lemme \ref{prec} un $z\in X$ qui s'envoie sur $t^{w(s)}e_2$ et tel que $\varphi(z)=p^{w(s)}\delta_2(p)z$.
   Comme $\sigma_a(z)-a^{w(s)}\delta_2(a)z\in L\cdot e_1$
  et comme $\varphi$ commute \`{a} $\Gamma$, on obtient facilement que $\sigma_a(z)=
  a^{w(s)}\delta_2(a)$, donc $z$ est propre pour $\varphi$ et $\Gamma$, ce qui contredit
 le lemme \ref{propre}.

$\bullet$ Si $w(s)\in\{0,-1,-2,...\}$, le lemme \ref{prec} montre que
  $(\nabla-w(\delta_1))X\subset X_2\subset \mathcal{R}e_1$, donc l'image de $X$ dans
  $\mathcal{R}(\delta_2)$ est contenue dans $\{fe_2| \nabla f-w(s)f=0\}$.
  Le dernier espace est nul si $w(s)<0$ et de dimension $1$ si $w(s)=0$.
  Enfin, $$X\cap \mathcal{R}e_1=\{fe_1| \nabla(\nabla f)+w(s)\nabla f=0\}$$
  et ceci est de dimension $2$ (resp. $1$) si $w(s)<0$ (resp. $w(s)=0$), toujours
  d'apr\`{e}s le lemme \ref{eqdiff} (si $w(s)=0$, noter que la relation
  $\nabla(\nabla f)=0$ force $\nabla f\in L\cap t\cdot \mathcal{R}=\{0\}$ et donc $f\in L$). Ceci permet de conclure.

\end{proof}

\section{L'involution $w_D$}\label{invol}

  On d\'{e}crit l'involution
  $w_D$ sur $D_{\rm rig}^{\psi=0}$ dans le cas o\`{u} $V$ est trianguline
et on d\'{e}montre le th\'{e}or\`{e}me \ref{th2}. 

 \subsection{Le module $\mathcal{R}\boxtimes_{\delta}\p1$} \label{delta}

Soit $\delta: \qpet\to L^*$ un caract\`{e}re continu. Rappelons
     que $\mathcal{R}^{\psi=0}$ est libre de rang $1$ sur\footnote{ L'anneau $\mathcal{R}(\Gamma)$ est d\'{e}fini par
$\mathcal{R}(\Gamma)=\Lambda(\Gamma)\otimes_{\Lambda(\Gamma_2)} \mathcal{R}(\Gamma_2)$, o\`{u} $\Gamma_2=\chi^{-1}(1+p^2\zp)$,
$\Lambda(G)$ est l'alg\`{e}bre des mesures \`{a} valeurs dans $O_L$ sur le groupe de Lie $p$-adique $G$ et
$\mathcal{R}(\Gamma_2)$ est d\'{e}fini comme l'anneau $\mathcal{R}$, en rempla\c{c}ant la variable $T$ par
$\gamma-1$, pour n'importe quel g\'{e}n\'{e}rateur topologique $\gamma$ de $\Gamma_2$.}
 $\mathcal{R}(\Gamma)$, de base $1+T$ (cela d\'{e}coule
     de \cite[cor. V.1.13]{C5}).
   L'involution $i_{\delta}$ de
     $L[\Gamma]$ qui envoie $\sigma_a$ sur $\delta(a)\sigma_{\frac{1}{a}}$ se prolonge
     de mani\`{e}re unique en une involution $i_{\delta}$ de $\mathcal{R}(\Gamma)$ (\cite[lemme V.2.3]{C5}) et
     on d\'{e}finit $w_{\delta}(f\cdot (1+T))=i_{\delta}(f)\cdot (1+T)$ si $f\in\mathcal{R}(\Gamma)$. Cela fournit une involution continue sur
     $\mathcal{R}^{\psi=0}$ telle que pour toute distribution $\mu$ sur $\zpet$ on ait
     $$w_{\delta}\left(\int_{\zpet} (1+T)^x \mu\right)=w_{\delta}\left(\left(\int_{\zpet} \sigma_{x} \mu\right)\cdot (1+T)\right)=$$
     $$\int_{\zpet} \delta(x)\sigma_{\frac{1}{x}}\mu\cdot (1+T)=\int_{\zpet} \delta(x)
     (1+T)^{\frac{1}{x}}\mu.$$ 
Donc l'involution $w_{\delta}$ que l'on vient de d\'{e}finir co\"\i ncide avec celle d\'{e}finie dans
\ref{invdevis}.

   Soit ${\rm LA}(\p1(\delta))$ 
l'espace des fonctions
localement analytiques $\phi:\qp\to L$ telles que $x\to \delta(x)\phi\left(\frac{1}{x}\right)$
se prolonge en une fonction localement analytique, muni de l'action d\'{e}finie par 
(o\`{u} $g=\left(\begin{smallmatrix} a & b\\c & d\end{smallmatrix}\right)$)
  $$(g^{-1}\cdot \phi)(x)=\delta(cx+d)\phi\left(\frac{ax+b}{cx+d}\right).$$
Un exercice standard de la th\'{e}orie des induites paraboliques
identifie
${\rm Ind}_{B}^{G} (1, \delta)\otimes \delta^{-1}$ \`{a} ${\rm LA}(\p1(\delta))$
(en tant que ${\rm GL}_2(\qp)$-modules topologiques). 
L'application qui envoie $\mu$ sur $\left({\rm Res}_{\zp}(\mu),
{\rm Res}_{\zp}(w\cdot\mu)\right)$ compos\'{e}e avec l'isomorphisme d'Amice\footnote{ Cet isomorphisme identifie $\mathcal{R}^+$ \`{a} 
l'espace
des distributions sur $\zp$ via la transform\'{e}e d'Amice $\mu\to \sum_{n\geq 0} \int_{\zp} \binom{x}{n}\mu\cdot T^n$.} identifie 
le dual de ${\rm LA}(\p1(\delta))$ (comme espace vectoriel topologique) 
\`{a} $$\mathcal{R}^+\boxtimes_{\delta}\p1=\{(f_1,f_2)\in\mathcal{R}^+\times\mathcal{R}^+| {\rm Res}_{\zpet}(f_2)=
w_{\delta}({\rm Res}_{\zpet}(f_1))\}.$$ 
L'espace $\mathcal{R}^+\boxtimes_{\delta}\p1$ est ainsi muni d'une action de ${\rm GL}_2(\qp)$.
Cette action est donn\'{e}e par des formules explicites comme dans \cite[II.1]{C5} (voir aussi \cite[4.1]{C7}). Ces formules
permettent de prolonger l'action de ${\rm GL}_2(\qp)$ \`{a} $\mathcal{R}\boxtimes_{\delta}\p1$ (d\'{e}fini de la
m\^{e}me mani\`{e}re que $\mathcal{R}^+\boxtimes_{\delta}\p1$, en rempla\c{c}ant $\mathcal{R}^+$ par $\mathcal{R}$).

\subsection{D\'{e}vissage du module $D_{\rm rig}\boxtimes\p1$}

   On d\'{e}montre le th\'{e}or\`{e}me \ref{th2} de l'introduction et on l'utilise ensuite
pour d\'{e}visser le module $D_{\rm rig}\boxtimes\p1$. Rappelons que l'on suppose que $V$ est trianguline,
correspondant \`{a} un point $s=(\delta_1,\delta_2,\mathcal{L})$ de $\mathcal{S}_{\rm irr}$.
Rappelons aussi que $D_{\rm rig}$ est contenu dans $D_{\rm rig}\boxtimes\p1$ (voir le paragraphe \ref{LLpadique}).
Le r\'{e}sultat suivant est crucial pour la suite:

\begin{lemma}\label{coro}
 On a $w\cdot e_1\in (D_{\rm rig}\boxtimes\p1)^U$.
\end{lemma}

\begin{proof}
 
  C'est une cons\'{e}quence de la proposition \ref{annulation} et du fait que 
$w\cdot e_1=(0,e_1)$ et $e_1\in X$.
\end{proof}

  \begin{proposition}\label{inv}
   Pour tout $f\in\mathcal{R}^{\psi=0}$ on a $w_{D}(f\cdot e_1)=\delta_1(-1)w_{\delta_D\delta_1^{-2}}(f)\cdot e_1$.
  \end{proposition}

  \begin{proof} On laisse au lecteur le soin de
 v\'{e}rifier l'identit\'{e} suivante (dans laquelle $w=\left(\begin{smallmatrix} 0 & 1 \\1 & 0\end{smallmatrix}\right)$)
  $$ \left(\begin{smallmatrix} 1 & 1\\0 & 1\end{smallmatrix}\right)w\cdot\left(\begin{smallmatrix} -1 & 0\\0 & -1\end{smallmatrix}\right)\cdot
  \left(\begin{smallmatrix} -1 & 0 \\0 & 1\end{smallmatrix}\right)\cdot
  \left(\begin{smallmatrix} 1 & 1 \\0 & 1\end{smallmatrix}\right)\cdot w=
 w\cdot
  \left(\begin{smallmatrix} 1 & 1 \\0 & 1\end{smallmatrix}\right).
  $$

   Appliquons cette identit\'{e} \`{a} $e_1\in D_{\rm rig}\boxtimes
  \p1$. Le terme de gauche est \'{e}gal\footnote{Utiliser le fait que $w\cdot e_1=(0,e_1)$, le lemme
  \ref{coro}, l'\'{e}galit\'{e} $\sigma_{-1}(e_1)=\delta_1(-1)e_1$ et enfin le fait que
  $D_{\rm rig}\boxtimes \p1$ a pour caract\`{e}re central $\delta_D$.}
  \`{a} $\delta_1(-1)(1+T)e_1$, ce qui permet donc d'\'{e}crire
  $w_D((1+T)e_1)=\delta_1(-1)(1+T)e_1$. Comme $\sigma_a(e_1)=\delta_1(a)e_1$
et comme $w_D(\sigma_a(z))=\delta_D(a)\sigma_{\frac{1}{a}}(w_D(z))$,
il est facile de voir que les applications $F(f)=w_D(fe_1)$ et
$G(f)=\delta_1(-1)w_{\delta_D\cdot\delta_1^{-2}}(f)\cdot e_1$
sont semi-lin\'{e}aires pour l'action de $i_{\delta_D\cdot \delta_1^{-1}}$.
Comme elles co\"\i ncident sur $1+T$, qui est une base de $\mathcal{R}^{\psi=0}$ sur
$\mathcal{R}(\Gamma)$, on obtient bien $F=G$, d'o\`{u} le r\'{e}sultat.

 \end{proof}

   Soit $\mathcal{R}e_1\boxtimes\p1=(D_{\rm rig}\boxtimes\p1)\cap (\mathcal{R}e_1\times\mathcal{R}e_1)$.

  \begin{corollary} $\mathcal{R}e_1\boxtimes\p1$ est un sous-module ferm\'{e} de $D_{\rm rig}\boxtimes \p1$, stable sous
   l'action de ${\rm GL}_2(\qp)$. 
  \end{corollary}

\begin{proof}
  En tant qu'espace vectoriel topologique, $\mathcal{R}e_1\boxtimes\p1$ (resp. $D_{\rm rig}\boxtimes\p1$) s'identifie 
\`{a} $\mathcal{R}e_1\times\mathcal{R}e_1$ (resp. $D_{\rm rig}\times D_{\rm rig}$). La fermeture de $\mathcal{R}e_1\boxtimes\p1$
dans $D_{\rm rig}\boxtimes \p1$ suit donc de celle de $\mathcal{R}e_1$ dans $D_{\rm rig}$. La stabilit\'{e} de $D_{\rm rig}\boxtimes \p1$
sous l'action de ${\rm GL}_2(\qp)$ d\'{e}coule de la proposition \ref{inv}, de la stabilit\'{e} de $\mathcal{R}e_1$ par $\varphi$ et 
$\Gamma$, et des formules donnant
 l'action de ${\rm GL}_2(\qp)$ sur $D_{\rm rig}\boxtimes\p1$.

\end{proof}

  \begin{proposition}\label{inv2}

 Pour tout $B\in \mathcal{R}^{\psi=0}$ on a $$p_s(w_D(B\cdot \varphi(\hat{e}_2)))=\delta_2(-1)w_{\delta_D\cdot\delta_2^{-2}}(B)\varphi(e_2).$$
   
 \end{proposition}

  \begin{proof} Un argument de semi-lin\'{e}arit\'{e} comme dans la preuve de la proposition
  \ref{inv} montre que l'on peut supposer que $B=1+T$.

   Posons $Y=(D_{\rm rig}\boxtimes \p1)/(\mathcal{R}e_1\boxtimes \p1)$ et notons $z\to [z]$ la projection canonique
 $D_{\rm rig}\boxtimes\p1\to Y$. Nous aurons besoin du r\'{e}sultat suivant, analogue du lemme \ref{coro}.

 \begin{lemma}\label{involagain}
   L'\'{e}l\'{e}ment $[w\cdot \varphi(\hat{e}_2)]$ de $Y$ est invariant par $U$.
  \end{lemma}

  \begin{proof}
   Le lemme \ref{ut} montre qu'il suffit
de v\'{e}rifier que $\dim_L Y^{u^+=0}<\infty$ et que
   $[w\cdot \varphi(\hat{e}_2)]\in Y^{u^+=0}$.
    Le $L$-espace vectoriel $$W=\{f\in\mathcal{R}(\delta_2)|
    (\nabla-w(\delta_1))(\nabla-w(\delta_2))f=0\}$$ est de dimension finie d'apr\`{e}s la
    proposition \ref{fini}. Si $z=(z_1,z_2)\in D_{\rm rig}\boxtimes \p1$ satisfait $[z]\in Y^{u^+=0}$, alors $u^+z\in\mathcal{R}e_1\boxtimes \p1$
et on d\'{e}duit de la proposition \ref{inf}
    que $z_1\in \mathcal{R}e_1$ et que $p_s(z_2)\in W$. Donc
     $$z=z_1+w\cdot {\rm Res}_{p\zp}(z_2)\equiv
    w\cdot {\rm Res}_{p\zp}(p_s(z_2))\pmod {\mathcal{R}e_1\boxtimes\p1},$$
     ce qui montre que $\dim_L Y^{u^+=0}<\infty$.

      Pour conclure, il nous reste \`{a} v\'{e}rifier que $u^+(w\cdot \varphi(\hat{e}_2))\in \mathcal{R}e_1\boxtimes \p1$.
      Cela d\'{e}coule de la proposition \ref{inf} et du fait que $(\nabla-w(\delta_2))\hat{e}_2\in \mathcal{R}e_1$ (car
      $\sigma_a(\hat{e}_2)-\delta_2(a)\hat{e}_2\in \mathcal{R}e_1$ pour tout $a\in \zpet$).

  \end{proof}

       Revenons \`{a} la preuve de la proposition \ref{inv2}. 
On applique l'identit\'{e} matricielle du d\'{e}but de la preuve de la proposition \ref{inv} \`{a} $[w\cdot \varphi(\hat{e}_2)]$.
       Noter que $[w\cdot \varphi(\hat{e}_2)]$ est vecteur propre pour l'op\'{e}rateur $\left(\begin{smallmatrix} -1 & 0 \\0 & 1\end{smallmatrix}\right)$,
     de valeur propre $\delta_2\cdot \delta_D(-1)$. L'identit\'{e} matricielle s'\'{e}crit donc $p_s(w_D((1+T)\varphi(\hat{e}_2)))=\delta_2(-1)(1+T)\varphi(e_2)$, ce qui permet
     de conclure.

  \end{proof}

 \begin{corollary}\label{dev}
   La suite exacte $0\to\mathcal{R}e_1\to D_{\rm rig}\to \mathcal{R}e_2\to 0$ induit une suite exacte
de ${\rm GL}_2(\qp)$-modules topologiques
    $$0\to (\mathcal{R}\boxtimes_{\delta_D\cdot\delta_1^{-2}}\p1)\otimes\delta_1\to D_{\rm rig}\boxtimes\p1\to
 (\mathcal{R}\boxtimes_{\delta_D\cdot
    \delta_2^{-2}}\p1)\otimes\delta_2\to 0.$$
  \end{corollary}

  \begin{proof}  Commen\c{c}ons par d\'{e}finir les morphismes dans cette suite exacte. 
L'application $i$
de $\mathcal{R}\boxtimes_{\delta_D\cdot\delta_1^{-2}}\p1$ dans $D_{\rm rig}\boxtimes\p1$ envoie
$(f_1,f_2)$ sur $(f_1\cdot e_1, \delta_1(-1)f_2\cdot e_1)$. L'application ${\rm}$ de
$D_{\rm rig}\boxtimes\p1$ dans $\mathcal{R}\boxtimes_{\delta_D\cdot\delta_2^{-2}}\p1$
envoie $(A_1\cdot e_1+B_1\cdot\varphi(\hat{e}_2), A_2\cdot e_1+B_2\cdot \varphi(\hat{e}_2))$
sur $(B_1,\delta_2(-1)B_2)$, o\`{u} $A_i, B_i\in\mathcal{R}$ ($\hat{e}_2\in D_{\rm rig}$
est un rel\`{e}vement fix\'{e} de $e_2$). Le fait que ces applications $i$ et ${\rm pr}$ sont bien
d\'{e}finies et induisent une suite exacte d'espaces vectoriels topologiques est une cons\'{e}quence imm\'{e}diate
des propositions \ref{inv} et \ref{inv2}. La ${\rm GL}_2(\qp)$-\'{e}quivariance (\`{a} torsion par
$\delta_1$, resp. $\delta_2$ pr\`{e}s) suit des propositions \ref{inv} et \ref{inv2}, du fait que
$f\to f\cdot e_1$ et $p_s$ sont des morphismes de $(\varphi,\Gamma)$-modules et des formules explicites
donnant l'action de ${\rm GL}_2(\qp)$ sur les modules intervenant dans la suite exacte.

  \end{proof}

\begin{remark}
 Soit $\delta:\qpet\to L^*$ un caract\`{e}re continu. On a vu dans \ref{delta} que 
l'on a un isomorphisme de ${\rm GL}_2(\qp)$-modules topologiques
 $$\mathcal{R}^+\boxtimes_{\delta}\p1\simeq {\rm Ind}_{B}^{G}(1,\delta)^*\otimes\delta\simeq
({\rm Ind}_B^{G}(\delta^{-1}\otimes 1))^*.$$
On peut v\'{e}rifier (en utilisant des arguments identiques \`{a} ceux du chapitre II de \cite{C5}; voir aussi le paragraphe 4.1 de \cite{C7})
que l'application $\mathcal{R}\boxtimes_{\delta}\p1\to {\rm Ind}_{B}^{G}(\chi^{-1}\delta\otimes \chi^{-1})$
d\'{e}finie par $$z\to \phi_z, \quad \phi_z(g)={\rm res}_0\left({\rm Res}_{\zp}(wgz)\frac{dT}{1+T}\right)$$
est une surjection ${\rm GL}_2(\qp)$-\'{e}quivariante, de noyau $\mathcal{R}^+\boxtimes_{\delta}\p1$. On dispose
donc d'une suite exacte de ${\rm GL}_2(\qp)$-modules topologiques $$0\to ({\rm Ind}_B^{G}(\delta^{-1}\otimes 1))^*
\to \mathcal{R}\boxtimes_{\delta}\p1\to {\rm Ind}_B^{G}(\chi^{-1}\delta\otimes\chi^{-1})\to 0.$$
Combin\'{e} au corollaire \ref{dev} et \`{a} la suite exacte $$0\to (\Pi^{\rm an})^*\otimes\delta_D\to
D_{\rm rig}\boxtimes\p1\to \Pi^{\rm an}\to 0,$$ cela permet de montrer que $\Pi^{\rm an}$ est de longueur finie
et que $$(\Pi^{\rm an})^{\rm ss}=({\rm Ind}_{B}^{G}(\delta_1\otimes\chi^{-1}\delta_2))^{\rm ss}
\oplus ({\rm Ind}_B^{G}(\delta_2\otimes\chi^{-1}\delta_1))^{\rm ss}.$$
Il faut travailler un peu plus \cite{C7} pour d\'eterminer
les extensions entre les constituants de Jordan-H\"older.

\end{remark}

\end{document}